\documentclass[12pt, a4paper, reqno]{amsart}
\usepackage[utf8]{inputenc}
\usepackage[english]{babel}

\usepackage{amsrefs}
\usepackage{lmodern}
\usepackage{upgreek}
\usepackage{color}
\usepackage{amsfonts}
\usepackage{amsmath}
\usepackage{amssymb}
\usepackage{amsthm}

\usepackage{pdfpages}

\usepackage{setspace}

\renewcommand{\Re}{\operatorname{Re}}

\usepackage{hyperref}
\hypersetup{
    colorlinks=true,
    linkcolor=blue,
    filecolor=blue,      
    urlcolor=blue,
    citecolor=blue,
    pdftitle={},
    pdfpagemode=FullScreen,
    }

\numberwithin{equation}{section}
\newtheorem{theorem}{Theorem}[section]
\newtheorem{lem}[theorem]{Lemma}
\newtheorem{thm}{Theorem}

\newtheorem{prop}[theorem]{Proposition}

\newtheorem{hyp}{Hypothesis}

\theoremstyle{definition}

\newtheorem{rem}{Remark}[section]

\newcommand{\Mod}[1]{\ (\mathrm{mod}\ #1)}

\title[Binary correlations]{On binary correlations of Fourier coefficients of holomorphic cusp forms at prime arguments}

\author{Jiseong Kim and Kunjakanan Nath}

\address{Department of Mathematical Sciences, McNeese State University, Lake Charles, LA 70609, USA}
\email{jiseongk51@gmail.com}

\address{Institut \'Elie Cartan de Lorraine, Universit\'e de Lorraine, CNRS, F-54000 Nancy, France}
\email{kunjakanan@gmail.com}

\date{}

\begin{document}

\begin{abstract}
Let $\{\lambda_f(n)\}_{n \geq 1}$ be the normalized Hecke eigenvalues of a given holomorphic cusp form $f$ of even weight $k$. We show under the assumption of the existence of Littlewood's type zero-free region for $L(s, f, \chi)$, where $\chi$ is a Dirichlet character modulo $q$, that if $X^{2/3+\varepsilon} \ll H \ll X^{1-\varepsilon}$ with $\varepsilon>0$, then for any $A\geq 1$,
$$\sum_{1\leq |h|\leq H}\bigg| \sum_{\substack{X<n,\: m \leq 2X \\ n - m = h}} 
\lambda_f(n)\Lambda(n)\lambda_f(m)\Lambda(m) \bigg|^2
\ll_{A} \frac{HX^2}{(\log X)^{A}}$$
holds. Moreover, under an additional hypothesis on the fourth moment of certain Dirichlet polynomials (which follows from GRH for $L(s, f)$), we show that the above result can be strengthened to hold in a wider range 
$X^{1/3+\varepsilon}\ll H \ll X^{1-\varepsilon}$. Finally, if we average over the forms $f$, then for $X^{\varepsilon}\ll H\ll X^{1-\varepsilon}$ and for any $A\geq 1$,
 $$
\sum_{f\in \mathcal{H}_k}\omega_f\sum_{1\leq |h|\leq H}\bigg| \sum_{\substack{X<n,\:  m \leq  2X \\ n - m = h}} 
\lambda_f(n)\Lambda(n)\lambda_f(m)\Lambda(m) \bigg|^2
\ll_{A}\frac{HX^2}{(\log X)^{A}},$$
where $\mathcal{H}_k$ is the Hecke basis for the space of holomorphic cusp forms of weight $k$ for the full modular group $\mathrm{SL}(2, \mathbb{Z})$ and $\omega_f$ are {\it harmonic weights} associated with $f\in \mathcal{H}_k$. These results may be viewed as modular analogues of the averaged forms of the Hardy–Littlewood prime tuple conjecture.

\end{abstract}


\maketitle

\section{Introduction}

The study of correlations of arithmetic functions has been a fascinating area of research due to its connections to many famous open problems in number theory, such as the Twin Prime Conjecture, the Chowla Conjecture, to name a few. The strong form of the celebrated Hardy-Littlewood conjecture states that for any $\varepsilon>0$, one has\footnote{Here and throughout, $\Lambda$ denotes the von Mangoldt function, that is, $\Lambda(n)=\log p$ if $n=p^k$ for some prime $p$ and integer $k\geq 1$, and $\Lambda(n)=0$ otherwise.}
$$
\sum_{X < n \leq 2X} \Lambda(n) \Lambda(n + h) = \mathfrak{S}(h) X + O_{\varepsilon}\left(X^{1/2+\varepsilon}\right),
$$ 
where $\mathfrak{S}(h)$ is a singular series given by
\begin{align*}
    \mathfrak{S}(h)=
    \begin{cases}
        2\prod_{p\geq 3}\bigg(1-\dfrac{1}{(p-1)^2}\bigg)\prod_{\substack{p|h\\ p\geq 3}}\bigg(\dfrac{p-1}{p-2}\bigg), & 2|h,\\
        0, & 2\nmid h.
    \end{cases}
\end{align*}
In fact, no result is currently known even when the error term is weakened to $o(X)$. However, the averaged versions over shifts $1 \leq h \leq H$ for $H = X^{\theta+\varepsilon}$ with various admissible values of $\theta$ have been studied by Wolke \cite{W1989}, Mikawa \cite{Mi1991}, Perelli-Pintz \cite{PP1993}, and the best result so far is $\theta = 8/33$ as shown by Matom\"aki, Radziwi\l{}\l{}, and Tao in their recent nice work \cite{MRT2019i}. For more progress on this topic involving $k$-tuples, we invite interested readers to see the recent works of Matom\"aki-Shao-Tao-Ter\"av\"ainen \cite{MXTJ2023} and Matom\"aki-Radziwi\l{}\l{}-Shao-Tao-Ter\"av\"ainen \cite{MMXTJ2024}. 

 In this paper, we are interested in a certain modular analogue of the Hardy-Littlewood type correlation problem. In order to address our problem, we begin with the following set-up. Let $k\geq 1$ be an even integer. Denote by $\mathcal{H}_{k}$ the Hecke basis for the space of holomorphic cusp forms of weight $k$ for the full modular group $\Gamma=\mathrm{SL}(2, \mathbb{Z}).$ It is known that
 \begin{align*}
     |\mathcal{H}_k|=\dfrac{k}{12}+O(k^{2/3}).
 \end{align*}
 Given $f\in \mathcal{H}_k$, its Fourier expansion can be written in the form
 \[f(z)=\sum_{n=1}^\infty\lambda_f(n)n^{(k-1)/2}e(nz) \qquad (\text{Im}(z)>0),\]
 where as usual $e(z)=e^{2\pi i z}$. The Fourier coefficients $\lambda_f(n)$ are the normalized eigenvalues of the Hecke operators and satisfy the Hecke relation
 \begin{align*}
     \lambda_f(m)\lambda_f(n)=\sum_{d|(m, n)}\lambda_f\bigg(\dfrac{mn}{d^2}\bigg)
 \end{align*}
 for all positive integers $m, n$. Moreover, it also satisfies the following bound due to Deligne
 \begin{align*}
     |\lambda_f(n)|\leq d_2(n),
 \end{align*}
 where $d_2$ is the divisor function. These standard facts can be found, for example, in \cite[Chapter 14]{IKbook}.

Given any arithmetic function $a\colon \mathbb{N}\to \mathbb{C}$, an important question in analytic number theory is to study correlations of the form  
\[
\sum_{X<n\leq 2X} \lambda_{f}(n)a(n+h).
\] 
This has been extensively studied in the literature, see, for example, works of Assing-Blomer-Li \cite{ABL2021}, Jiang-L\"u \cite{JL2024}, Pitt \cite{Pit2013}. We also mention the work of Lou \cite{M2019}, who studied the average version of the sum of the type
$$
\sum_{|h| \leq H} \bigg| \sum_{\substack{X<n,\: m \leq 2X\\ n-m=h}} \lambda_f(n) \Lambda(m) \bigg|^2.
$$
Moreover, in recent times, several works have also investigated \emph{murmurations}, concerning the distribution of Hecke eigenvalues at prime arguments within families from various perspectives (see, for example, the works of Bober-Booker-Lee-Lowry-Duda \cite{B2025} and Kundu-Mueller \cite{K2025}). It is worth noting that, when considering families in the weight aspect, some of these results are obtained under the assumption of the Generalized Riemann Hypothesis, as they involve values restricted to prime arguments. Also, note that the range of each aspect (such as weight or level) is a crucial factor in these results.

Our goal in this paper is to investigate the averaged correlations of the sequence $\{\lambda_f(n)\Lambda(n)\}_{n\geq 1}.$ In particular, for various ranges of $H$ in terms of $X$, we will consider the following quantity
\[V_{f}(X; H):= \sum_{1\leq |h|\leq H}\bigg|\sum_{\substack{X<n,\:  m \leq 2X\\ n-m=h}}\lambda_f(n)\Lambda(n)\lambda_f(m)\Lambda(m)\bigg|^2.\]

Our first main result gives an estimate for the average correlation under the assumption of the following {Littlewood type zero-free region}. 

\begin{hyp}[Littlewood type zero-free region]\label{zerofree1}
For any $\varepsilon>0$ and for any Dirichlet character $\chi \pmod{q}$, the twisted $L$-function
\begin{equation}
   L(s, f, \chi) := \sum_{n=1}^{\infty}\frac{\lambda_{f}(n)\chi(n)}{n^{s}} \qquad (\Re(s)>1)
\end{equation}
has no zeros for $s=\sigma+it$ with $|t|$ large and 
\[\sigma\geq 1-\dfrac{(\log \log (q|t|))^{1+\varepsilon}}{\log (q|t|)}.
\]
\end{hyp}

\begin{rem}
It is known (see, for example, \cite[Theorem 5.39]{IKbook}) that there exists $c_f>0$ such that there are no zeros of $L(s, f, \chi)$  for $s=\sigma + it$  with $|t|$ large and 
\[\sigma\geq 1 -\dfrac{c_{f}}{\log (qk^{2}|t|)},\]
except possibly a one simple real zero.

\end{rem}

\begin{thm}\label{unconditionalthm} Assume Hypothesis \ref{zerofree1}. Let $\varepsilon>0$ and let $X$ be sufficiently large. Suppose that $X^{2/3+\varepsilon}\leq H  \leq X^{1-\varepsilon}$. Then, for any $A\geq 1$, we have
    \begin{align*}
       V_{f}(X; H)\ll_{A, \: k} \dfrac{HX^2}{(\log X)^A}.
    \end{align*}
\end{thm}

\begin{rem}
It is known that modular forms of the \emph{CM type} satisfy the Vinogradov-Korobov type zero-free region (see, for example, work of Coleman \cite{Col1990}). So, in this case, one does not need to assume Hypothesis \ref{zerofree1}.
\end{rem}

The work of Perelli \cite{P1984} would imply a version of Theorem \ref{unconditionalthm} in the range $X/(\log X)^{O(1)}\leq H\leq X$ without the assumption of Hypothesis \ref{zerofree1}. However, the key difficulty in our work is to take $H$ as small as possible, which requires the assumption on zero-free region.

Note that we have a restriction of $H
\geq X^{2/3+\varepsilon}$ in Theorem \ref{unconditionalthm}. It is reasonable to expect that the above theorem should hold for $H\geq X^\varepsilon$. To make further progress toward improving the range of $H$, we assume the following hypothesis on the fourth‐moment estimates.
\begin{hyp}[Fourth moment estimates]\label{Hypothesis}
For any \(T, M \geq 1\) with \(\min(T,M)\gg1\) and for any integer $q\geq 1$, suppose that
\begin{align}
\sideset{}{^*}\sum_{\chi \bmod q}
\int_{T/2}^{T}
\Big|
  \sum_{M\le n\le2M}\dfrac{\lambda_f(n)\chi(n)}{n^{1/2+it}}
\Big|^4
\,dt
&\ll_k
qT\bigl(\log(qTM)\bigr)^{O(1)},
\label{eq:4th2}
\end{align}
where $^*$ denotes the sum over primitive Dirichlet characters.
\end{hyp}

\begin{rem}
    One can use integration by parts and Hypothesis \ref{Hypothesis} to deduce (see, for example, \cite[Lemma 2.2]{MRT2019i}) that for any \(T, M \geq 1\) with \(\min(T,M)\gg1\) and for any integer $q\geq 1$,
\begin{align*}
    \sideset{}{^*}\sum_{\chi \bmod q}
\int_{T/2}^{T}
\Big|
  \sum_{M\le n\le2M}\dfrac{\lambda_f(n)(\log n)\chi(n)}{n^{1/2+it}}
\Big|^4
\,dt
&\ll_k
qT\bigl(\log(qTM)\bigr)^{O(1)}.
\end{align*}
\end{rem}

\begin{thm}\label{conditionalthm} Assume Hypotheses \ref{zerofree1} and \ref{Hypothesis}.
 Let $\varepsilon>0$ and let $X$ be sufficiently large. Suppose that $X^{1/3+\varepsilon}\leq H  \leq X^{1-\varepsilon}.$
Then, for any $A\geq 1$, we have
    \begin{align*}
       V_{f}(X; H)\ll_{A, \: k} \dfrac{HX^2}{(\log X)^A}.
    \end{align*}
\end{thm}

The key point of Theorem \ref{conditionalthm} is the quantitative improvement over Theorem \ref{unconditionalthm}, allowing us to take $H\geq X^{1/3+\varepsilon}$ instead of $H\geq X^{2/3+\varepsilon}$. However, Theorem \ref{conditionalthm} has the disadvantage that it has a stronger requirement on the fourth moment estimate. But the assumption in Hypothesis \ref{Hypothesis} is quite natural and can be deduced from the bound
$$\sideset{}{^*}\sum_{\chi \bmod q}\int_{T/2}^{T} \left| L(1/2+it, f, \chi)\right|^{4}dt \ll_{k} qT (\log qT)^{O(1)},$$
which is expected to be true. This can be viewed as an analogue of the eighth moment of the Dirichlet $L$-function, which is well studied when one averages over the moduli $q$ (see, for example, \cite[Proposition 3.2]{CLMR2024}). It is also worth mentioning in this context the work of Bettin and Conrey \cite{BC2021}, who studied averages of the Dirichlet polynomials in the context of the Riemann zeta function and demonstrated many nice applications together with valuable insights into the moment conjectures in random matrix theory. 

\begin{rem}
    It seems that under Generalized Riemann Hypothesis (GRH), we can extend the range of $H$ to $X^{\varepsilon} \le H \le X^{1-\varepsilon}$.
This is because, under GRH, Lemma~\ref{Lemma: major arc for small q} holds for $H \ge X^{\varepsilon}$, and, for any integer $k\geq 1$, it is known that  (see, for example, \cite{MC2014})
\[
\sideset{}{^*}\sum_{\chi \bmod q}\int_{T/2}^{T} \bigl|L\!\left(\tfrac12+it, f, \chi\right)\bigr|^{k^{2}}\,dt
\;\ll_{k}\; qT\,\bigl(\log(qT)\bigr)^{O_{k}(1)}.
\] 
\end{rem}

Finally, we present our result on average over the family of holomorphic cusp forms, where we can take $H\geq X^\varepsilon$.

\begin{thm}\label{averagetheorem} Let $\varepsilon>0$ and let $X$ be sufficiently large. Suppose that \( X^{{\varepsilon}} \leq H \leq X^{1-\varepsilon} \) and \( k \gg X^{1+5\varepsilon} \) is an even positive integer. Then, for any $A\geq 1$, we have
\begin{align*}
    \sum_{f\in \mathcal{H}_k}\omega_fV_f(X; H)\ll_A \dfrac{HX^2}{(\log X)^A},
\end{align*}
where $\omega_f$ are {harmonic weights} associated with $f\in \mathcal{H}_k$ given by
\begin{align*}
    \omega_f:=\dfrac{\Gamma(k-1)}{(4\pi)^{k-1}\langle f, f\rangle}.
\end{align*}
Here $\langle f, f \rangle$ denotes the Petersson inner product and $\Gamma$ is the Gamma function.
\end{thm}

\begin{rem}
Applying the Petersson trace formula (see Lemma \ref{Lemma: Petersson} below) directly yields a similar type of result in Theorem \ref{averagetheorem}, but only in the more restricted range $k \gg X^{2+\varepsilon}$. 
\end{rem}

\begin{rem}
Lately, there has been some work on convolutions of Hecke eigenvalues and divisor functions, such as 
\[\lambda_{f_{1}} \ast \lambda_{f_{2}} \ast \cdots \ast \lambda_{f_{k}} \ast \underbrace{1 \ast 1 \ast \cdots \ast 1}_{\ell\text{ times}}\]
see, for example, the work of Huang and L\"u \cite{HG2026}.
Attaching $\Lambda$ transforms the associated Dirichlet polynomials into a linear combination such as 
\[
\sum_{\ell=1}^{k} \sum_{n} \lambda_{f_{\ell}}(n)\Lambda(n) + \ell \sum_{n} \Lambda(n).
\]
Therefore, one can see that our results can be generalized directly to these cases.
\end{rem}

\subsection{Sketch of the proof}
The first step in our proof of Theorem \ref{unconditionalthm}
is to apply the Hardy–Littlewood circle method, so that
$$V_{f}(X; H)=\sum_{1\leq |h|\le H}
\left|
  \int_{0}^{1}
    \lvert S_{f,\: \Lambda}(\alpha)\rvert^{2}
    \,e(-h\alpha)\,
    d\alpha
\right|^{2},
$$
where throughout this paper, we set
\[S_{f, \: \Lambda}(\alpha):=\sum_{X<n\leq  2X}\Lambda(n)\lambda_f(n)e(n\alpha).\] 
We dissect the interval $[0,1]$ into \emph{major arcs} and \emph{minor arcs} based on the applicability of the pointwise bounds for $S_{f,\: \Lambda}(\alpha)$, which provides good upper bounds when $\alpha$ is close to a rational number $a/q$ with sufficiently large $q$ (see, for example, Lemma~\ref{Lemma: Exponential sum L infinity}).

For the \emph{major arcs}, we apply Gallagher’s lemma to reduce the problem to estimating the quantity
$$
\int_{X}^{2X}\left|
\sum_{x< n\leq x+Y} \Lambda(n)\,\lambda_f(n)\,
e\!\left(\frac{na}{q}\right)
\right|^2\: dx
$$
for some suitable $1\ll Y\ll X$. The next crucial observation is that, due to the presence of $\Lambda(n)$, we can replace $e(na/q)$ with character sums over Dirichlet characters $\chi$ modulo $q$ much more efficiently.
  For small $q$, we transfer the short sum with additive rational weights to a short sum with Dirichlet character weights and then apply the classical approach based on zero-density approach and Littlewood's type zero-free regions, this leads us to assume Hypothesis \ref{zerofree1}. For large $q$, we instead use bilinear form estimates \`a la Vinogradov. On the other hand, for the \emph{minor arcs}, we use exponential sum estimates (see Lemma \ref{Lemma: Exponential sum L infinity}).

For Theorem \ref{conditionalthm}, we use the same ingredients for estimating the major arc contribution as in the proof of Theorem \ref{unconditionalthm}.  However, for the minor arcs, we follow the arguments of Matom\"aki, Radziwi\l{}\l{}, and Tao \cite{MRT2019i} to further improve the range of $H$ from $H\geq X^{2/3+\varepsilon}$ to $H\geq X^{{1/3}+\varepsilon}$. In particular, we reduce our problem to understanding the mean value of the \emph{Dirichlet polynomials} (see Proposition \ref{Prop: Minor thm2} below). This in turn requires us to assume Hypothesis \ref{Hypothesis}. 

Finally, in the case of Theorem~\ref{averagetheorem}, the Hardy--Littlewood circle method framework allows us to handle shorter summation ranges via Gallagher's lemma, which is particularly advantageous when applying the Petersson trace formula. Note that the explicit Wilton bound (see, for example, \cite[Theorem 1.1]{G2013}) is given by
\[
\sum_{n \le y} \lambda_{f}(n)e(n\alpha) 
\ll_{\varepsilon} y^{1/2+\varepsilon}k^{1/2+\varepsilon} 
\]
for any $\varepsilon>0$ and $\alpha\in \mathbb{R}$. However, the above explicit bound is not suitable for our purpose when considering averages over $f \in \mathcal{H}_{k}$; we instead apply the Petersson trace formula to obtain alternative estimates. In particular, we will show that if $k\gg X^{1+5\varepsilon}$ and $H\gg X^{{\varepsilon}}$ with $\varepsilon>0$, then for any real number $\alpha$,
\begin{align*}
    \sum_{f\in \mathcal{H}_k}\omega_f\bigg|\sum_{X<n\leq X+H}\lambda_f(n)\Lambda(n)e(n\alpha)\bigg|^2\ll_{\varepsilon}H^2X^{-\varepsilon/2}. 
\end{align*}

\subsection{Notation} We will use standard notation throughout the paper.
\begin{itemize}
    \item Given functions $f, g\colon \mathbb{R}\to \mathbb{C}$, the expressions of the form $f(x)=O(g(x))$, $f(x) \ll g(x)$, and $g(x) \gg f(x)$ signify that $|f(x)| \leq c|g(x)|$ for all sufficiently large $x$, where $c>0$ is an absolute constant. A subscript of the form $\ll_A$ means the implied constant may depend on the parameter $A$. The notation $f(x) \asymp g(x)$ indicates that $f(x) \ll g(x) \ll f(x)$. 
    \item We also let $o(1)$ denote a quantity that tends to zero as $x \rightarrow \infty$. 
\item Given any real number $x$, we write $e(x)=e^{2\pi ix}$.
\item For any two arithmetic functions $g, h\colon \mathbb{N}\to \mathbb{C}$, we define their Dirichlet convolution by
\[
(g * h)(n) := \sum_{d \mid n} g(d) h\left(\frac{n}{d}\right).
\]
\item Given any integers $n\geq 1$ and $m\geq 2$, $d_m(n)$ denotes $m$-fold divisor function, that is, $d_m(n)=\underbrace{(1*\cdots *1)}_{m\: \text{times}}(n)$ and $\omega(n)$ denotes the number of distinct prime divisors of $n$. 
\item We set $(a, b)$ to be the greatest common divisor of integers $a$ and $b$.
\item For $g\colon \mathbb{N}\to \mathbb{C}$, the $\ell^2$-norm of $g$ is defined as 
\[
\|g\|_{\ell^{2}} := \bigg( \sum_{n\in \mathbb{N}} |g(n)|^{2} \bigg)^{1/2}.
\]
\item We let $1_{\mathcal{S}}$ be the characteristic function of the set $\mathcal{S}$, that is, $1_{\mathcal{S}} (x) = 1$ if $x\in \mathcal{S}$ and $0$, otherwise.
\end{itemize}

The $L$-function associated with $\lambda_f$ is given by 
\begin{align*}
    L(s, f):=\sum_{n\geq 1}\dfrac{\lambda_f(n)}{n^s}\qquad (\Re(s)>1).
\end{align*}
Throughout, we define $\Lambda_f$ and $\mu_f$ as the coefficients of $-L^\prime(s, f)/L(s, f)$ and $1/L(s, f)$, respectively, that is,
\begin{align}\label{Def: 1/L}
    \dfrac{-L^\prime(s, f)}{L(s, f)}:=\sum_{n\geq 1}\dfrac{\Lambda_f(n)}{n^s} \quad \text{and} \quad \dfrac{1}{L(s, f)}:=\sum_{n\geq 1}\dfrac{\mu_f(n)}{n^s} \quad (\Re(s)>1).
\end{align}
It is known that $|\Lambda_f(n)|\ll \log n$ and $|\mu_f(n)|\leq d_2(n)$ (see, for example, \cite{P1984}). Also, note that $\Lambda_f$ is supported on prime powers.

\subsection{Organization of the paper} The paper is organized as follows. Section \ref{sec: Prelim} is devoted to establishing a few preliminary results. In Section \ref{sec: Proof of unconditional theorem}, we prove Theorem \ref{unconditionalthm}. Next, we establish Theorem \ref{conditionalthm} in Section \ref{sec: Proof of theorem conditional}. Then, we will prove Theorem \ref{averagetheorem} in Section \ref{sec: Proof of theorem averaged}. Finally, we end with a concluding remark in Section \ref{sec: concluding}.

\section{Preliminaries}\label{sec: Prelim}
In this section, we establish a few preliminary results that will be used in the proofs of our main theorems. 
We begin with the following lemma, which is an easy consequence of Gallagher's lemma.

\begin{lem}\label{prop1}
Let $X$ be large. Let $1\ll Y \ll X$ and $|\alpha-a/q|\leq {(1+o(1))}/(2Y)$ for some integers $(a, q)=1$. Then, we have
\begin{equation}\label{eq:Talpha}
\int_{\alpha-1/Y}^{\alpha+1/Y}\bigg|\sum_{X<n\leq 2X}\Lambda(n)\lambda_f(n)e(n\beta)\bigg|^2\:d\beta \ll \frac1{Y^2}
\int_{X}^{2X}
\Bigl\lvert
\sum_{u< n\leq u+Y}\Lambda(n)\,\lambda_f(n)\,
e(an/q)
\Bigr\rvert^2
\,du.
\end{equation}
\end{lem}

\begin{proof} 
Since $|\alpha-a/q|\leq (1+o(1))/(2Y)$, we have that
\[
\Big[\alpha-\frac1Y,\;\alpha+\frac1Y\Big]
\subset
\Bigl[\frac aq-\frac{3}{2Y},\;\frac aq+\frac{3}{2Y}\Bigr].
\]
Hence, the claim now follows immediately from Gallagher’s lemma \cite[Lemma 1.9]{M1971}.
\end{proof}

Next, in order to deal with the Dirichlet polynomials in Theorem \ref{conditionalthm}, we will need the mean value theorem.

\begin{lem}[Mean value theorem]\label{Lemma: mean value}
Let $f\colon \mathbb{N}\to \mathbb{C}$ be an arithmetic function supported on $[X/2, 4X]$ for some $X\geq 2$. Then, for all $T>0$ and $T_0\in \mathbb{R}$, we have
\begin{align*}
    \int_{T_0}^{T_0+T}\bigg|\sum_{n}\dfrac{f(n)}{n^{1/2+it}}\bigg|^2\: dt\ll \dfrac{T+X}{X}\|f\|_{\ell^2}^2,
\end{align*}
and for any integer $q\geq 2$,
\begin{align*}
    \sum_{\chi\Mod q}\int_{T_0}^{T_0+T}\bigg|\sum_{n}\dfrac{f(n)\chi(n)}{n^{1/2+it}}\bigg|^2\: dt\ll \dfrac{qT + X}{X}\|f\|_{\ell^2}^2(\log qTX)^3.
\end{align*}
    \end{lem}
    \begin{proof}
        For the proof of the first part, see \cite[Theorem 9.1]{IKbook}, and for the second part, see \cite[Theorem 9.12]{IKbook}.
    \end{proof}

    In order to handle the minor arcs in Theorem \ref{conditionalthm}, we need the Heath-Brown identity for $\Lambda_f$, which we state in the following lemma.

    \begin{lem}[Heath-Brown's decomposition]\label{Lemma: Heath Brown}
    Let $B\geq 1$ and $0<\varepsilon<1/3$ be fixed. Let $X\geq 2$ and let ${H}$ be such that $X^{1/3+\varepsilon}\leq H\leq X$. Let $1\leq q_0\leq (\log X)^B$ be an integer. Then, we can decompose $\Lambda_f(q_0n)1_{(X/q_0, \: 2X/q_0]}(n)$ as a linear combination  of $O_{\varepsilon, \: B} ((\log X)^{O(1)})$ components $\widetilde{\Lambda_f}$, each of which is one of the following types:
    \begin{itemize}
     \item Type $d_1, d_2$ sum: A function of the form
$$
\widetilde{\Lambda_f}=\left(\mathfrak{f} * \mathfrak{g}_1 *\mathfrak{g}_2 \right) 1_{\left(X / q_0,\:  2X / q_0\right]}
$$
for some arithmetic functions $\mathfrak{f}, \mathfrak{g}_1, \mathfrak{g}_2\colon \mathbb{N} \rightarrow \mathbb{C}$ where $\mathfrak{f}$ is supported on $[N, 2N]$ and $|\mathfrak{f}|\leq d_{10}$, and $\mathfrak{g}_1, \mathfrak{g}_2$ is either of the form $1_{\left(M_j, 2M_j\right]}\cdot \lambda_{f} \cdot \chi$ or $  1_{\left(M_j, 2 M_j\right]}\cdot \lambda_{f} \cdot \chi \cdot \log $ satisfying the bounds
$$
\begin{gathered}
1 \ll N <_{k, \varepsilon} X^{\varepsilon^2}, \quad HX^{-\varepsilon^2}\ll M_1\ll M_2\ll X/q_0, \quad \text{and} \quad 
NM_1M_2\asymp_{k, \varepsilon_0} X/q_0.
\end{gathered}
$$
\item Type II sums: A function of the form
$$
\widetilde{\Lambda_f}=(\mathfrak{f} * \mathfrak{g}) 1_{\left(X / q_0,\: 2 X / q_0\right]}
$$
for some arithmetic functions $\mathfrak{f}, \mathfrak{g}\colon \mathbb{N} \rightarrow \mathbb{C}$ such that $|\mathfrak{f}|, |\mathfrak{g}|\ll d_{10}$ and supported on $[N, 2 N]$ and $[M, 2 M]$ respectively, satisfying the bounds
$$
X^{\varepsilon^2} \ll N \ll HX^{-\varepsilon^2}\quad \text{and} \quad NM\asymp_{k, \varepsilon_0}X/q_0.
$$
\item Small sum: A function $\widetilde{\Lambda_f}$ supported on $(X/q_0, 2X/q_0]$ satisfying
\begin{align}\notag 
    \|\widetilde{\Lambda_f}\|_{\ell^2}^2\ll_{k, \varepsilon} X^{1-\varepsilon^2/8}.
\end{align}
 \end{itemize}
\end{lem}
\begin{proof}
 For any $z\geq 2$, write
\begin{align*}
\mathcal{M}(s, f; z):=\sum_{n\leq z}\dfrac{\mu_f(n)}{n^s},
\end{align*}
 where $\mu_f$ is given by \eqref{Def: 1/L}. Then, we have
 \begin{align*}
     L(s, f)\mathcal{M}(s, f; z) = 1 + \sum_{n>z}\dfrac{a_f(n; z)}{n^s}
 \end{align*}
 for some complex coefficients $a_f(n; z)$. Since for any integer $L\geq 1$, we have
 \begin{align*}
     \dfrac{L^\prime}{L}(s, f) \big(1-L(s, f)\mathcal{M}(s, f; z)\big)^L=\dfrac{L^\prime}{L}(s, f) + \sum_{\ell=1}^L(-1)^\ell \binom{L}{\ell}L^\prime(s, f)L(s, f)^{\ell-1}\mathcal{M}(s, f; z)^\ell,
 \end{align*}
 we can now conclude that if $n\leq 2z^L$, then
 \begin{align*}
     \Lambda_f(n)=\sum_{\ell=1}^L(-1)^{\ell-1}\binom{L}{\ell}\sum_{\substack{\prod_{j=1}^\ell m_jn_j=n\\ n_1, \dotsc, n_\ell\leq z}}\lambda_f(m_1)(\log m_1)\lambda_f(m_2)\cdots \lambda_f(m_\ell)\mu_f(n_1)\cdots \mu_f(n_\ell).
 \end{align*}
 One can now follow the arguments exactly as in the proofs of \cite[Lemmas 2.15, 2.16]{MRT2019i} to establish the claim.
 \end{proof}

\subsection{Zero density estimates and $\Lambda_f(n)$ in almost all short intervals.} 
We will use the following zero-density estimates to handle certain major arcs with small $q$.
\begin{lem}\label{density}
Let $1/2\leq \sigma\leq 1$ and let $T\geq 2$. Define
 $$N_f(\sigma, T, \chi) :=\#\Bigl\{ \rho = \beta + i\gamma\colon 
      L(\rho, f, \chi) = 0,\; \beta \geq \sigma,\; |\gamma| \leq T \Bigr\}.$$
Then, we have
\begin{align*}
\sum_{\chi\Mod q}N_f(\sigma, T, \chi)\ll_{k}
\begin{cases}
(qT)^{\frac{4(1-\sigma)}{3-2\sigma}}(\log qT)^{O(1)}, & 1/2\leq \sigma\leq 5/6,\\
(qT)^{3(1-\sigma)}(\log qT)^{O(1)}, & 5/6\leq \sigma \leq 1.
\end{cases}
\end{align*}
\end{lem}

\begin{proof}
    See \cite[Theorem 1]{K1998}.
\end{proof}

\begin{lem}\label{Lemma: major arc for small q}
Assume Hypothesis \ref{zerofree1}. Let $\varepsilon>0$, $A\geq 1$, $q\leq (\log X)^{2A+O(1)}$, and $X^{{1/3}+\varepsilon}\leq H\leq X^{1-\varepsilon}$. Then, for any $B>0$, we have
\begin{align*}
   \dfrac{1}{X} \int_{X}^{2X}\sum_{\chi \Mod q}\bigg|\sum_{x<n\leq x+H}\Lambda_f(n)\chi(n)\bigg|^2\: dx \ll \dfrac{H^2}{(\log X)^B}.
\end{align*}
    \end{lem}
    
\begin{proof}
By an explicit formula (see \cite[Eq. (17)]{P1984}), for any $1\leq T\leq x$, we have
\begin{align*}
\sum_{x<n\leq x+H}\Lambda_f(n)\chi(n)=-\sum_{|\text{Im}(\rho)|\leq T}\dfrac{(x+H)^\rho - x^\rho}{\rho} + O_{k}\bigg(\dfrac{x\log^2qx}{T} + x^{1/4}\log qx\bigg),
\end{align*}
where the sum is over the non-trivial zeros of $L(s, f, \chi)$. So, we have
\begin{align*}
\int_{X}^{2X}\sum_{\chi \Mod q}\bigg|\sum_{x<n\leq x+H}\Lambda_f(n)\chi(n)\bigg|^2\: dx \ll &\: \int_{X}^{2X}\sum_{\chi\Mod q}\bigg|\sum_{|\text{Im}(\rho)|\leq T}\dfrac{(x+H)^\rho - x^\rho}{\rho}\bigg|^2\ dx\\
& +  O_{k}\bigg(\dfrac{X^3\varphi(q)(\log qX)^3}{T^2} + \varphi(q)X^{3/2}(\log qX)^2\bigg).
\end{align*}
Expanding the square and performing the integration, we have
\begin{align*}
\int_X^{2X}\sum_{\chi\Mod q}&\bigg|\sum_{|\text{Im}(\rho)|\leq T}\dfrac{(x+H)^\rho - x^\rho}{\rho}\bigg|^2\ dx\\
&=\sum_{\chi\Mod q}\sum_{|\text{Im}(\rho_1)|\leq T}\sum_{|\text{Im}(\rho_2)|\leq T}\dfrac{1}{\rho_1\bar{\rho_2}}\int_{X}^{2X}\big((x+H)^{\rho_1}-x^{\rho_1}\big)\big((x+H)^{\bar{\rho_2}}-x^{\bar{\rho_2}}\big)\: dx.
\end{align*}
For simplicity, let us put $H=\delta x$ so that
\begin{align*}
\int_{X}^{2X}&\big((x+H)^{\rho_1}-x^{\rho_1}\big)\big((x+H)^{\bar{\rho_2}}-x^{\bar{\rho_2}}\big)\: dx\\
&=\big((1+\delta)^{\rho_1}-1\big)\big((1+\delta)^{\bar{\rho_2}}-1)\dfrac{(2X)^{\rho_1+\bar{\rho_2}+1}-X^{\rho_1 + \bar{\rho_2}+1}}{\rho_1+\bar{\rho_2}+1}.
\end{align*}
So, we have
\begin{align*}
\int_X^{2X}&\sum_{\chi\Mod q}\bigg|\sum_{|\text{Im}(\rho)|\leq T}\dfrac{(x+H)^\rho - x^\rho}{\rho}\bigg|^2\ dx\\
&=\sum_{\chi\Mod q}\sum_{|\text{Im}(\rho_1)|\leq T}\sum_{|\text{Im}(\rho_2)|\leq T}\dfrac{\big((1+\delta)^{\rho_1}-1\big)\big((1+\delta)^{\bar{\rho_2}}-1)}{\rho_1\bar{\rho_2}}\bigg(\dfrac{(2X)^{\rho_1+\bar{\rho_2}+1}-X^{\rho_1 + \bar{\rho_2}+1}}{\rho_1+\bar{\rho_2}+1}\bigg).
\end{align*}
Note that
\begin{align*}
\bigg|\dfrac{(1+\delta)^\rho-1}{\rho}\bigg|=\bigg|\int_{1}^{1+\delta}t^{\rho-1}\: dt\bigg|\leq \int_{1}^{1+\delta}t^{\text{Re}(\rho)-1}\: dt\ll \delta
\end{align*}
and $|X^{\rho_1 + \bar{\rho_2}}|\ll X^{2\text{Re}(\rho_1)} + X^{2\text{Re}(\rho_2)}$. Therefore, by symmetry, the expression above is
\begin{align*}
\ll \delta^2 X\sum_{\chi\Mod q}\sum_{|\text{Im}(\rho_1)|\leq T}X^{2\text{Re}(\rho_1)}\sum_{|\text{Im}(\rho_2)|\leq T}\dfrac{1}{|\rho_1+\bar{\rho_2}+1|}\ll_{k} \frac{H^{2}}{X}(\log X)^{2}\sum_{\chi\Mod q}\sum_{|\text{Im}(\rho)|\leq T}X^{2\text{Re}(\rho)},
\end{align*}
where we have used the fact that there are $\ll_{k} \log (q(j+1))$ zeros $\rho$ with $|\text{Im}(\rho)|\in [j, j+1]$.
Therefore, by Hypothesis \ref{zerofree1} and Lemma \ref{density},
we have
\begin{align*}
\sum_{\chi\Mod q}\sum_{|\text{Im}(\rho)|\leq T}X^{2\text{Re}(\rho)}
\ll_{k} X^{2}\bigg(\dfrac{(qT)^{3+\varepsilon}}{X^2}\bigg)^{1-\sigma_1}.
\end{align*}
where $\sigma_{1}:= 1- \frac{(\log \log qT)^{1+\varepsilon}}{\log (qT)}.$
By choosing $T =X^{2/3-\varepsilon/2},$ the proof is complete. 
\end{proof}

    \subsection{Exponential sum estimates \`a la Vinogradov}

    \begin{lem}\label{Lemma: Exponential sum L infinity}
    Let $\alpha=a/q+\beta$ for some $(a, q)=1$ and $|\beta|\leq 1/q^2$. Let $x^{2/3}(\log x)^{O(1)}\leq H\leq x$. Then, we have
        \begin{align*}
\sum_{x \leq n \leq x+H} \lambda_{f}(n)\Lambda(n)e(n\alpha)
   \ll_{k}
   \left(
      \frac{H}{\sqrt{q}}
      + x^{1/2}q^{1/2}
      + H^{1/2}x^{1/3}
      + x^{2/3}
   \right)(\log x)^{O(1)}.
\end{align*}
    \end{lem}

\begin{proof}
We will closely follow Perelli's argument in the proof of \cite[Theorem 1]{P1984} to establish the result. Note that
\begin{align*}
   \sum_{x \leq n \leq x+H} \lambda_{f}(n)\Lambda(n)e(n\alpha)= \sum_{x \leq n \leq x+H} \Lambda_f(n)e(n\alpha) + O((\log x)^4).
\end{align*}
Therefore, it is enough to estimate the sum $S$ given by
\begin{align}\notag
    S:=\sum_{x \leq n \leq x+H} \Lambda_f(n)e(n\alpha)
\end{align}
By Vaughan's identity for $\Lambda_f$ (see, for example, \cite[p. 154]{P1984}), we have
\begin{align}\label{Eq: sum S decompose}
S=S_1+S_2-S_3+S_4,
\end{align}
where for some parameters $U, V\geq 2$ to be chosen later,
\begin{align*}
&S_1 =\sum_{\substack{x<n\leq x+H\\ n\leq U}}\Lambda_f(n)e(n\alpha),\\
&S_2=\sum_{\substack{x<mn\leq x+H\\ m\leq V}}\mu_f(m)\lambda_f(n)(\log n)e(mn\alpha),\\
&S_3=\sum_{\substack{x<mn\leq x+H\\ m\leq UV}}a_f(m)\lambda_f(n)e(mn\alpha),\\
&S_4=\sum_{\substack{x<mn\leq x+H\\ m>U, n>V}}\Lambda_f(m)b_f(n)e(mn\alpha).
\end{align*}
Here for integers $m, n\geq 1$,
\[|\Lambda_f(n)|\ll \log n, \quad  |\mu_f(m)|\ll d_2(m), \quad  |a_f(m)|\ll (\log m)d_3(m), \quad  |b_{f}(n)|\ll d_4(n).\]
Since $|\Lambda_f(n)|\ll \log n$, we bound the sum $S_1$ trivially to obtain
\begin{align}\label{Eq: S1}
S_1\ll U\log x.
\end{align}
Next, for $S_2$, by Wilton's bound
\begin{align}\label{Eq: Wilton bound}
    \sum_{n \leq y} \lambda_{f}(n)e(n\theta) \ll_{k} y^{1/2}\log y
\end{align} 
for any real numbers $y\geq 2$ and $\theta$, we have
\begin{align}\notag 
|S_2| &\ll_{k} \sum_{m\leq V}|\mu_f(m)|\bigg|\sum_{x/m<n\leq (x+H)/m}\lambda_f(n)(\log n)e(mn\alpha)\bigg|\\
\label{Eq: S2} &\ll_{k}  \sum_{m\leq V}d_2(m)(\log x)(x/m)^{1/2}\ll_{k} x^{1/2}V^{1/2}(\log x)(\log V),
\end{align}
by using bounds for the divisor function. We further write
\begin{align}\label{Eq: S3 split}
S_3=\sum_{\substack{x<mn\leq x+H\\ m\leq V}}a_f(m)\lambda_f(n)e(mn\alpha) + \sum_{\substack{x<mn\leq x+H\\ V<m\leq UV}}a_f(m)\lambda_f(n)e(mn\alpha)=S_3^\prime + S_3^{\prime \prime},
\end{align}
say. Then, the Wilton bound \eqref{Eq: Wilton bound} implies that
\begin{align}\label{Eq: S3prime}
|S_3^\prime|\ll_{k} x^{1/2}(\log x)(\log V)\sum_{m\leq V}\dfrac{d_3(m)}{m^{1/2}}\ll_{k} x^{1/2}V^{1/2}(\log x)(\log V)^3.
\end{align}
Choose $U=V=x^{1/3}$. It is enough to estimate the sum
\begin{align*}
\sum_{\substack{x<mn\leq x+H\\ m, n>x^{1/3}}}\alpha(m)\beta(n)e(mna/q),
\end{align*}
where $|\alpha(m)|\leq d_4(m)$ and $|\beta(n)|\leq \log n$. Using bilinear sum estimate (see, for example, the proof of \cite[Lemma 5.5]{MMS2017}), we have
\begin{align*}
\sum_{\substack{x<mn\leq x+H\\ m, n>x^{1/3}}}\alpha(m)\beta(n)e(mna/q)\ll \bigg(H^{1/2}x^{1/3} + x^{1/2}q^{1/2} + \dfrac{H}{\sqrt{q}}\bigg)(\log x)^{O(1)}.
\end{align*}
This implies that
\begin{align}\label{Eq: S3primeprime and S4}
|S_3^{\prime\prime}|, |S_4|\ll_{k} \bigg(H^{1/2}x^{1/3} + x^{1/2}q^{1/2} + \dfrac{H}{\sqrt{q}}\bigg)(\log x)^{O(1)}.
\end{align}
Since $U=V=x^{1/3}$, from \eqref{Eq: S1}, \eqref{Eq: S2}, and \eqref{Eq: S3prime}, we have
\begin{align}\label{Eq: final}
|S_1|\ll x^{1/3}(\log x), \quad |S_2|, |S_3^\prime|\ll_{k} x^{2/3}(\log x)^4.
\end{align}
Combining the above estimate \eqref{Eq: final} together with \eqref{Eq: sum S decompose}, \eqref{Eq: S3 split}, and \eqref{Eq: S3primeprime and S4} completes the proof.
\end{proof}

\section{Proof of Theorem \ref{unconditionalthm}}\label{sec: Proof of unconditional theorem}
The goal of this section is to prove Theorem \ref{unconditionalthm}. Throughout this section, for any $\varepsilon>0$ and sufficiently large $A$, we set
\[
X^{2/3+\varepsilon} \ll H \ll X^{1-\varepsilon},
\quad Q = X^{1/3} (\log X)^{A/3 +O(1)}, \quad \text{and} \quad R=H(\log X)^{-2A}.
\] 
Let 
\begin{align}\label{Def: Major arcs Thm 1}
\mathfrak{M}
:= \bigcup_{q\le Q}\;\bigcup_{\substack{a=1\\(a,q)=1}}^{q}
\Bigl(\frac{a}{q} - \frac{1}{qR},\;
      \frac{a}{q} + \frac{1}{qR}\Bigr),
\end{align}
denote the major arcs and let 
\begin{align}\label{Def: Minor arcs Thm 1}
  \mathfrak m := [0,1]\setminus\mathfrak{M}  
\end{align}
denote the minor arcs. We also recall that for any $\alpha\in \mathbb{R}$,
\begin{align*}
    S_{f, \: \Lambda}(\alpha)=\sum_{X<n\leq 2X}\Lambda(n)\lambda_f(n)e(n\alpha).
\end{align*}

We start by establishing the following proposition over major arcs.

\begin{prop}[Major arcs estimate]\label{Major1} 
Assume that Hypothesis \ref{zerofree1} holds. Let $\varepsilon>0$ and let $X$ be large. Suppose that $X^{2/3+\varepsilon}\ll H\ll X^{1-\varepsilon}$. Then, we have
\[\sup_{\alpha\in\mathfrak{M}} \int_{\mathfrak{M}\: \cap\:  [\alpha-1/2H,\: \alpha +1/2H]} \left|S_{f,\: \Lambda}(\beta)\right|^{2} d\beta \ll_{A,k} \dfrac{X}{(\log X)^{A+2}},\]
where $\mathfrak{M}$ is given by $\eqref{Def: Major arcs Thm 1}$.
\end{prop}

\begin{proof}
    If $\alpha\in \mathfrak{M}$, then there exist $(a, q)=1$ with $q\leq Q$ such that
\begin{align*}
    \bigg|\alpha-\dfrac{a}{q}\bigg|\leq \dfrac{1}{qR}.
\end{align*}
Now we consider two cases, depending on whether $q\gg (\log X)^{2A+O(1)}$ or $q\ll (\log X)^{2A+O(1)}$.

We start by handling the first case. Since $q\gg (\log X)^{2A+ O(1)}$ in this case, we have $qR\geq H$.  Therefore, by Lemma \ref{prop1}, we have
\begin{align*}
    \int_{\mathfrak{M}\: \cap\:  [\alpha-1/2H,\: \alpha +1/2H]} \left|S_{f,\Lambda}(\beta)\right|^{2} d\beta \ll \dfrac{1}{H^2}\int_{X}^{2X}\bigg|\sum_{x<n\leq x+H}\Lambda(n)\lambda_f(n)e(na/q)\bigg|^2\: dx.
\end{align*}
So, it is enough to show that
\begin{align*}
    \int_{X}^{2X}\bigg|\sum_{x<n\leq x+H}\Lambda(n)\lambda_f(n)e(na/q)\bigg|^2\: dx\ll_{{A}} \dfrac{H^2X}{(\log X)^{{A+2}}}.
\end{align*}
%Now we consider two cases: depending on whether $q\ll (\log X)^{2A+O(1)}$ or $q\gg (\log X)^{2A+O(1)}$.
Since $X^{2/3+\varepsilon}\ll H\leq X^{1-\varepsilon}$ and $(\log X)^{2A+O(1)}\ll q\leq Q$ by Lemma \ref{Lemma: Exponential sum L infinity}, we have
\begin{align*}
    \sum_{x<n\leq x+H}\Lambda(n)\lambda_f(n)e(na/q)\ll_{{A}} \dfrac{H}{(\log X)^{{A+O(1)}}}.
\end{align*}
Hence,
\begin{align*}
    \int_{X}^{2X}\bigg| \sum_{x<n\leq x+H}\Lambda(n)\lambda_f(n)e(na/q)\bigg|^2\: dx\ll_{{A}} \dfrac{H^2X}{(\log X)^{{A+2}}},
\end{align*}
as desired.

We now handle the second case, that is, $q\ll (\log X)^{2A+O(1)}$. In this case, we have 
\[H\gg qR/(\log X)^{O(1)}\gg H(\log X)^{-O(1)}\gg X^{2/3+\varepsilon/2},\]
using the fact that $H\gg X^{2/3+\varepsilon}$. Therefore,
\begin{align*}
    \int_{\mathfrak{M}\: \cap \: [\alpha-1/2H,\: \alpha + 1/2H]}|S_{f, \: \Lambda}(\beta)|^2\: d\beta &\ll \int_{-(\log X)^{O(1)}/(qR)}^{(\log X)^{O(1)}/(qR)}|S_{f, \: \Lambda}(a/q+\beta)|^2\: d\beta\\
    &\ll \int_{-X^{-2/3-{\varepsilon/2}}}^{X^{-2/3-\varepsilon/2}}|S_{f, \: \Lambda}(a/q+\beta)|^2\:d\beta\\
    &\ll \dfrac{1}{X^{4/3+\varepsilon}}\int_{X}^{2X}\bigg|\sum_{x<n\leq x+X^{2/3+{\varepsilon/2}}}\Lambda(n)\lambda_f(n)e(na/q)\bigg|^2\: dx,
\end{align*}
where we have used Gallagher's lemma \cite[Lemma 1.9]{M1971} in the last line. So, it is enough to show that
\begin{align*}
    \int_{X}^{2X}\bigg|\sum_{x<n\leq x+X^{2/3+{\varepsilon/2}}}\Lambda(n)\lambda_f(n)e(na/q)\bigg|^2\: dx \ll_{{A}} \dfrac{X^{4/3+{\varepsilon}}\cdot X}{(\log X)^{{A+2}}}.
\end{align*}
Note that for $x\in [X, 2X]$, we have
\begin{align*}
    \sum_{x<n\leq x+X^{2/3+{\varepsilon/2}}}\Lambda(n)\lambda_f(n)e(na/q)=\sum_{x<n\leq x+X^{2/3+{\varepsilon/2}}}\Lambda_f(n)e(na/q) + O((\log X)^4).
\end{align*}
Then, by Fourier inversion on the group $\mathbb{Z}/q\mathbb{Z}$, we have
\begin{align*}
    e\bigg(\dfrac{an}{q}\bigg)\cdot 1_{(an, q)=1}=\dfrac{1}{\varphi(q)}\sum_{\chi\Mod q}\tau(\overline{\chi})\chi(an),
\end{align*}
where $\tau(\overline{\chi})$ is the Gauss sum. Therefore, we have
\begin{align*}
    \sum_{x<n\leq x+X^{2/3+{\varepsilon/2}}}\Lambda_f(n)e(na/q) =\dfrac{1}{\varphi(q)}\sum_{\chi\Mod q}\tau(\overline{\chi})\chi(a)\sum_{x<n\leq x+X^{2/3+{\varepsilon/2}}}\Lambda_f(n)\chi(n) + O((\log X)^4),
\end{align*}
where we have used the fact that $\Lambda_f$ is supported on prime powers to show that the contributions coming from $(n, q)>1$ is $\ll (\log X)^4$. This implies that
\begin{align*}
    \int_{X}^{2X} & \bigg|\sum_{x<n\leq x+X^{2/3+{\varepsilon/2}}}\Lambda(n)\lambda_f(n)e(na/q)\bigg|^2\: dx\\
    &\ll \dfrac{1}{\varphi(q)^2}\int_X^{2X}\bigg|\sum_{\chi\Mod q}\tau(\overline{\chi})\chi(a)\sum_{x<n\leq x+X^{2/3+{\varepsilon/2}}}\Lambda_f(n)\chi(n)\bigg|^2\: dx + O(X\log^8 X)\\
    &\ll (\log X)\int_{X}^{2X}\sum_{\chi\Mod q}\bigg|\sum_{x<n\leq x+X^{2/3+{\varepsilon/2}}}\Lambda_f(n)\chi(n)\bigg|^2\: dx + O(X\log^8 X),
\end{align*}
by Cauchy-Schwarz's inequality and the fact that $|\tau(\overline{\chi})|\leq \sqrt{q}$. Applying Lemma \ref{Lemma: major arc for small q}, we see that the above quantity is
\begin{align*}
    \ll_{{A}} \dfrac{X^{4/3+{\varepsilon}}\cdot X}{(\log X)^{{A+2}}} + X(\log X)^8\ll_{{A}} \dfrac{X^{4/3+{\varepsilon}}\cdot X}{(\log X)^{{A+2}}},
\end{align*}
as desired. This completes the proof.
\end{proof}

Next, we establish the estimates over the minor arcs.

\begin{prop}[Minor arcs estimate]\label{Minor1} Let $\varepsilon>0$ and let $X$ be large. Suppose that $X^{2/3+\varepsilon}\ll H\ll X^{1-\varepsilon}$. Then, we have
$$\sup_{\alpha\in \mathfrak{m}} \int_{\mathfrak{m} \: \cap\:  [\alpha-1/2H,\: \alpha +1/2H]} \left|S_{f,\Lambda}(\beta)\right|^{2} d\beta \ll_{A,k} \dfrac{\max\{HX(\log X)^{-2A}, X^{5/3}\} (\log X)^{O(1)}}{H},$$
where $\mathfrak{m}$ is given by \eqref{Def: Minor arcs Thm 1}.
\end{prop}

\begin{proof}
   Recall that $\mathfrak{m}$ is given by \eqref{Def: Minor arcs Thm 1}. The application of Dirichlet's approximation theorem implies that if $\theta \in \mathfrak m,$ then 
there exist integers $a$ and $q$ such that $(a,q)=1, Q \leq q \leq R$ and
$$\left|\theta - \frac{a}{q}\right| \leq \frac{1}{qR}.$$
Since $Q=X^{1/3}(\log X)^{A/3+O(1)}$, if $\theta\in \mathfrak{m}$, by applying Lemma \ref{Lemma: Exponential sum L infinity}, we have 
\begin{align*}
        S_{f,\: \Lambda}(\theta)=\sum_{X<n\leq 2X}\Lambda(n)\lambda_f(n)e(n\theta)\ll_{A,k} \max\{R^{1/2}X^{1/2}, X^{5/6}\} (\log X)^{O(1)}.
    \end{align*}
Applying  the above uniform bound, we immediately infer that
\begin{align*}
\sup_{\alpha\in \mathfrak{m}} \int_{\mathfrak{m}\cap [\alpha-1/2H, \alpha +1/2H]} \left|S_{f,\Lambda}(\beta)\right|^{2} d\beta  & \ll_{A,k} \max\{RX, X^{5/3}\} (\log X)^{O(1)}\sup_{\alpha\in \mathfrak{m}}\int_{\alpha-1/2H}^{\alpha+1/2H}d\beta\\
&\ll_{A,k} \dfrac{\max\{HX(\log X)^{-2A}, X^{5/3}\}(\log X)^{O(1)}}{H},    
\end{align*}
using the fact that $R=H/(\log X)^{2A}$. This completes the proof.
\end{proof}

We are now ready to prove Theorem \ref{unconditionalthm}.
\begin{proof}[Proof of Theorem \ref{unconditionalthm}] 
Using the orthogonality relation
\[
\int_{0}^{1} e(n\alpha) \, d\alpha =
\begin{cases}
1 & \text{if } n = 0, \\
0 & \text{otherwise},
\end{cases}
\]
we have
$$V_{f}(X; H)=\sum_{1\leq |h|\leq H}\bigg|\sum_{\substack{X<n, m\leq 2X\\ n-m=h}}\Lambda(n)\lambda_f(n)\Lambda(m)\lambda_f(m)\bigg|^2=\sum_{1\leq |h| \leq H} \left|\int_0^1 \left|S_{f,\: \Lambda}(\alpha)\right|^{2} e(-\alpha h)\: d\alpha \right|^{2}.$$
Now we apply smoothing on the sum over $h$ using \cite[Proposition 3.1]{MRT2019i}. To be precise,
let $\Phi\colon  \mathbb{R} \rightarrow \mathbb{R}^{+}$ be an even non-negative Schwartz function such that $\Phi(x) \geq 1$ for all $x \in[-1,1]$, and the Fourier transform \[\widehat{\Phi}(\xi):=\int_{-\infty}^\infty \Phi(x) e(-x \xi)\: dx\]
is supported on $[-1 / 2,1 / 2]$. 
Then 
\begin{align}\notag 
   \sum_{\left|h\right| \leq H} &\left| \int_{0}^1 \left|S_{f,\Lambda}(\alpha)\right|^{2} e( -\alpha h) d\alpha \right|^{2}  \\
  \notag  & \ll \sum_{h\in \mathbb{Z}}\left|\int_{0}^1 |S_{f,\Lambda}(\alpha)|^{2} e(\alpha h) d \alpha\right|^2 \Phi\left(\frac{h}{H}\right)\\
\label{Eq: Theorem unconditional split}   &\ll \sum_{h\in \mathbb{Z}}\left|\int_{\mathfrak{M}} |S_{f,\Lambda}(\alpha)|^{2} e(-\alpha h) d \alpha\right|^2 \Phi\left(\frac{h}{H}\right) + \sum_{h\in \mathbb{Z}}\left|\int_{\mathfrak{m}}|S_{f,\Lambda}(\alpha)|^{2} e(-\alpha h) d \alpha\right|^2 \Phi\left(\frac{h}{H}\right),
   \end{align} 
   where $\mathfrak{M}$ and $\mathfrak{m}$ are given by \eqref{Def: Major arcs Thm 1} and \eqref{Def: Minor arcs Thm 1}, respectively. Next, for $\mathcal{S}\in \{\mathfrak{M}, \mathfrak{m}\}$, we open the square and apply the Poisson summation formula, so that
\begin{equation}\begin{split}\notag 
\sum_{h\in \mathbb{Z}} &\left|\int_{\mathcal{S}}|S_{f,\Lambda}(-\alpha)|^{2} e(\alpha h) d \alpha\right|^2 \Phi\left(\frac{h}{H}\right)\\
&\ll \int_{\mathcal{S}}\left|S_{f,\Lambda}(\alpha)\right|^{2} \int_{\mathcal{S}}\left|S_{f,\Lambda}(\beta)\right|^{2}  \left|\sum_{h\in \mathbb{Z}} e((\alpha-\beta) h) \Phi\left(\frac{h}{H}\right) \right| d \beta\:  d \alpha 
\\
& \ll H \int_{\mathcal{S}}\left|S_{f,\Lambda}(\alpha)\right|^{2} \int_{\mathcal{S}}\left|S_{f,\Lambda}(\beta)\right|^{2} \left|\sum_{k\in \mathbb{Z}} \hat{\Phi}(H(\alpha-\beta+k))\right|d\beta \:  d\alpha.
\end{split}\end{equation}
Observe that by our assumption, $\hat{\Phi}(H(\alpha-\beta+k))$ vanishes unless $\beta \in [\alpha-1/2H, \alpha +1/2H].$ Combining the above estimate together with the Parseval identity, we deduce that
$$\sum_{|h| \leq H} \left|\int_{\mathcal{S}} \left|S_{f,\Lambda}(\alpha)\right|^{2} e(-\alpha h) \mathrm{d} \alpha \right|^{2}
\ll H\|\lambda_f\Lambda
\|_{\ell^{2}}^2 \sup_{\alpha\in \mathcal{S}} \int_{\mathcal{S}\: \cap\: [\alpha-1/2H, \alpha +1/2H]} |S_{f,\Lambda}(\beta)|^{2}d\beta,$$
where we have used the fact that $\mathcal{S}\subseteq [0,1]$ in the last line to apply the Parseval identity. Note that since ${|\lambda_f(n)\Lambda(n)| \ll \log n}$, we have
\begin{align*}
    \|\lambda_f \Lambda\|_{\ell^2}^2=\sum_{X<n\leq 2X}\lambda_f(n)^2\Lambda(n)^2\ll X (\log X)^2.
\end{align*}
So, we have
\begin{align}\label{Eq: Theorem unconditional final}
    \sum_{|h| \leq H} \left|\int_{\mathcal{S}} \left|S_{f,\: \Lambda}(\alpha)\right|^{2} e(-\alpha h) \mathrm{d} \alpha \right|^{2}\ll HX(\log X)^2\sup_{\alpha\in \mathcal{S}} \int_{\mathcal{S}\cap [\alpha-1/2H, \alpha +1/2H]} |S_{f,\Lambda}(\beta)|^{2}d\beta.
\end{align}
Finally, we combine \eqref{Eq: Theorem unconditional split}, \eqref{Eq: Theorem unconditional final} and then apply Propositions \ref{Major1} and \ref{Minor1} together with the fact that  $X^{2/3+\varepsilon}\leq H\leq X^{1-\varepsilon}$ to conclude that
\begin{align*}
    V_{f}(X; H)\ll_{A,k} \dfrac{HX^2}{(\log X)^A}.
\end{align*}
This completes the proof.
\end{proof}

\section{Proof of Theorem \ref{conditionalthm}} \label{sec: Proof of theorem conditional}
In this section, we will establish Theorem \ref{conditionalthm}. Let $\varepsilon>0$ and $A$ be a sufficiently large constant. For the rest of this section, we set  
\[
X^{{1/3+\varepsilon}} \ll H \ll X^{1-\varepsilon} \quad \text{and} \quad 
Q = (\log X)^{2A+O(1)}.
\]
Define the major arcs by 
\begin{align}\label{Def: major arcs thm2}
\mathfrak{M}^\prime
:= \bigcup_{q\le Q}\;\bigcup_{\substack{a=1\\(a,q)=1}}^{q}
\Bigl(\frac{a}{q} - \frac{(\log X)^{{3A}}}{H},\;
      \frac{a}{q} + \frac{(\log X)^{{3A}}}{H}\Bigr)
      \end{align}
and set the minor arcs as 
\begin{align}\label{Def: minor arcs thm2}
\mathfrak m^\prime := [0,1]\setminus\mathfrak M^\prime.
\end{align}

As in the proof of Theorem \ref{unconditionalthm} (see, for example, the relation \eqref{Eq: Theorem unconditional final}), it is enough to estimate the integrals over the major arcs and minor arcs. We begin with the major arcs.

\begin{prop}[Major arcs estimate] \label{Prop: Major thm2}
Assume that Hypothesis \ref{zerofree1} holds. Let $\varepsilon>0$ and let $X$ be large. Suppose that $X^{{1/3}+\varepsilon}\ll H\ll X^{1-\varepsilon}$. Then, we have 
\[\sup_{\alpha\in \mathfrak{M}^\prime} \int_{\mathfrak{M}^\prime\: \cap\:  [\alpha-1/2H,\:  \alpha +1/2H]} \left|S_{f,\: \Lambda}(\beta)\right|^{2} d\beta \ll_{A,k} \dfrac{X}{(\log X)^{{A+2}}},\]
where $\mathfrak{M}^\prime$ is given by \eqref{Def: major arcs thm2}.
\end{prop}

\begin{proof}
The proof is similar to the proof of Proposition \ref{Major1}. We will only highlight the necessary changes. Indeed, if $\alpha\in \mathfrak{M}^\prime$, then there exist $(a, q)=1$ with $q\leq (\log X)^{2A+O(1)}$ such that
\begin{align*}
    \bigg|\alpha-\dfrac{a}{q}\bigg|\leq \dfrac{(\log X)^{3A}}{H}.
\end{align*}
Therefore, by Lemma \ref{prop1}, we have
\begin{align*}
    \int_{\mathfrak{M}\: \cap\:  [\alpha-1/2H,\: \alpha +1/2H]} \left|S_{f,\Lambda}(\beta)\right|^{2} d\beta \ll \dfrac{(\log X)^{{6A}}}{H^2}\int_{X}^{2X}\bigg|\sum_{x<n\leq x+H(\log X)^{-{3A}}}\Lambda(n)\lambda_f(n)e(na/q)\bigg|^2\: dx.
\end{align*}
Arguing similarly as in the case of $q\ll (\log X)^{2A + O(1)}$ in the proof of Proposition \ref{Major1} together with the application of Lemma \ref{Lemma: major arc for small q} completes the proof.
\end{proof}

Next, we have the following minor arcs estimate.

\begin{prop}[Minor arcs estimate]\label{Prop: Minor thm2}
Assume that Hypothesis \ref{Hypothesis} holds. Let $\varepsilon>0$ and let $X$ be large. Suppose that $X^{{1/3}+\varepsilon}\leq H\leq X^{1-\varepsilon}$. Then, we have
$$\sup_{\alpha\in \mathfrak{m}^\prime} \int_{\mathfrak {m}^\prime\: \cap\:  [\alpha-1/2H, \: \alpha +1/2H]} \left|S_{f,\: \Lambda}(\beta)\right|^{2} d\beta \ll_{A,k} \dfrac{X}{(\log X)^{{A+2}}},$$
where $\mathfrak{m}^\prime$ is given by $\eqref{Def: minor arcs thm2}$.
\end{prop}

Before proving the above proposition, we
explain how to use this result together with Proposition \ref{Prop: Major thm2} to deduce Theorem \ref{conditionalthm}.

\begin{proof}[Proof of Theorem \ref{conditionalthm} assuming Proposition \ref{Prop: Minor thm2}]
The proof is similar to the proof of Theorem \ref{unconditionalthm}. Indeed, from \eqref{Eq: Theorem unconditional split} and \eqref{Eq: Theorem unconditional final} we have
\begin{align*}
   V_f(X; H)\ll &\:  HX(\log X)^2\bigg(\sup_{\alpha\in \mathfrak{M}^\prime}\int_{\mathfrak{M}^\prime\cap [\alpha-1/(2H), \alpha+1/(2H)]}|S_{f,\: \Lambda}(\beta)|^2\: d\beta  \\
   & + \sup_{\alpha\in \mathfrak{m}^\prime}\int_{\mathfrak{m}^\prime\cap [\alpha-1/(2H), \alpha+1/(2H)]}|S_{f,\: \Lambda}(\beta)|^2\: d\beta\bigg).
   \end{align*}
Next, applying Propositions \ref{Prop: Major thm2} and \ref{Prop: Minor thm2}, we see that the right-hand side of the above expression is $\ll HX^2/(\log X)^A$. This completes the proof.
\end{proof}

The rest of this section is devoted to establishing Proposition \ref{Prop: Minor thm2}. To prove this, we use ideas from the work of Matom\"aki, Radziwi\l{}\l{}, and Tao \cite{MRT2019i}.

\begin{proof}[Proof of Proposition \ref{Prop: Minor thm2}]
Recall that $\mathfrak{m}^\prime$ is given by \eqref{Def: minor arcs thm2}. If $\alpha\in \mathfrak{m}^\prime$, then for all $1\leq a\leq q\leq Q$ with $(a, q)=1$, we have 
\begin{align*}
    \dfrac{(\log X)^{{3A}}}{H}<\bigg|\alpha-\dfrac{a}{q}\bigg|\leq \dfrac{1}{qQ}.
\end{align*}
Write $\xi=\alpha-a/q$, so that $(\log X)^{{3A}}/H<|\xi|\leq 1/(qQ)$. Next, we have
\begin{align*}
    \int_{\mathfrak{m}^\prime\cap [\alpha-1/(2H), \alpha + 1/(2H)]}|S_{f, \: \Lambda}(\beta)|^2\: d\beta
    &\leq \int_{\xi-1/(2H)}^{\xi + 1/(2H)}|S_{f,\:  \Lambda}(a/q+\theta)|^2\: d\theta\\
    &\ll \int_{\xi-1/(2H)}^{\xi + 1/(2H)}\bigg|\sum_{X<n\leq 2X}\Lambda_f(n)e(na/q+n\theta)\bigg|^2\: d\theta + \dfrac{(\log X)^8}{H},
\end{align*}
where we have used the fact that $\Lambda_f$ is supported on prime powers in the last line. For the ease of notation, we write
\begin{align*}
    \mathcal{L}(1/2+iu, f, \chi, q_0; X):=\sum_{X<q_0\ell \leq 2X}\dfrac{\Lambda_f(q_0\ell)\chi(\ell)}{\ell^{1/2+iu}}.
\end{align*}
Our goal is to replace the exponential sums by Dirichlet polynomials. In order to do so, we apply \cite[Corollary 5.3]{MRT2019i} so that 
\begin{align*}
   \int_{\xi-1/(2H)}^{\xi + 1/(2H)} &\bigg|\sum_{X<n\leq 2X}\Lambda_f(n)e(na/q+n\theta)\bigg|^2\: d\theta\\ \ll &\:  \dfrac{d_2(q)^4}{|\xi|^2H^2q}\sup_{q=q_0q_1}\int_{\xi X/\sqrt{Q}}^{\xi X\sqrt{Q}}\bigg(\sum_{\chi\Mod {q_1}}\int_{t-2|\xi|H}^{t+2|\xi|H}\big|\mathcal{L}(1/2+iu, f, \chi, q_0; X)\big|\: du\bigg)^2\: dt\\
     & + \: \dfrac{1}{H^2}\bigg(\dfrac{1}{\sqrt{Q}} + \dfrac{1}{|\xi|H}\bigg)^2\int_{X}^{2X}\bigg(\sum_{x<n\leq x+2H}|\Lambda_f(n)|\bigg)^2\: dx.
\end{align*}
First, we handle the second term in the above expression. Using the fact that $|\Lambda_f(n)|\ll \log n$ together with $|\xi|H>(\log X)^{{3A}}$ and $Q=(\log X)^{2A+O(1)}$, we have
\begin{align*}
    \dfrac{1}{H^2}\bigg(\dfrac{1}{\sqrt{Q}} + \dfrac{1}{|\xi|H}\bigg)^2\int_{X}^{2X}\bigg(\sum_{x<n\leq x+2H}|\Lambda_f(n)|\bigg)^2\: dx
    &\ll \dfrac{1}{H^2} \cdot \dfrac{1}{(\log X)^{2A}}\cdot H^2X(\log X)^2\\
    &\ll \dfrac{X}{(\log X)^{{A+2}}},
\end{align*}
which is admissible. Also note that $d_2(q)^4\ll q^\varepsilon\ll (\log X)^{\varepsilon}$. Therefore, our task reduces to showing that if $(\log X)^{{3A}}/H<|\xi|\leq 1/(qQ)$, then
\begin{align*}
    \sup_{q=q_0q_1}\int_{\xi X/\sqrt{Q}}^{\xi X\sqrt{Q}}\bigg(\sum_{\chi\Mod {q_1}}\int_{t-2|\xi|H}^{t+2|\xi|H}\big|\mathcal{L}(1/2+iu, f, \chi, q_0; X)\big|\: du\bigg)^2\: dt \ll_{A}  \dfrac{q\xi^2H^2X}{(\log X)^{A+O(1)}}.
\end{align*}

By Lemma \ref{Lemma: Heath Brown}, we express $\Lambda_f(q_0\ell)1_{(X/q_0, 2X/q_0]}(\ell)$ as a linear combination of $O_{k, \: \varepsilon}((\log X)^{100})$ components $\widetilde{\Lambda_f}$, each of which is of one of the following types:
 \begin{itemize}
     \item Type $d_1, d_2$ sum: A function of the form
$$
\widetilde{\Lambda_f}=\left(\mathfrak{f} * \mathfrak{g}_1 *\mathfrak{g}_2 \right) 1_{\left(X / q_0,\:  2X / q_0\right]}
$$
for some arithmetic functions $\mathfrak{f}, \mathfrak{g}_1, \mathfrak{g}_2\colon \mathbb{N} \rightarrow \mathbb{C}$ where $\mathfrak{f}$ is supported on $[N, 2N]$ and $|\mathfrak{f}|\leq d_{10}$, and $\mathfrak{g}_1, \mathfrak{g}_2$ is either of the form $1_{\left(M_j, 2M_j\right]}\cdot \lambda_{f} \cdot \chi$ or $  1_{\left(M_j, 2 M_j\right]}\cdot \lambda_{f} \cdot \chi \cdot \log $ satisfying the bounds
$$
\begin{gathered}
1 \ll N <_{k, \varepsilon} X^{\varepsilon^2}, \quad HX^{-\varepsilon^2}\ll M_1\ll M_2\ll X/q_0, \quad \text{and} \quad 
NM_1M_2\asymp_{k, \varepsilon_0} X/q_0.
\end{gathered}
$$
\item Type II sums: A function of the form
$$
\widetilde{\Lambda_f}=(\mathfrak{f} * \mathfrak{g}) 1_{\left(X / q_0,\: 2 X / q_0\right]}
$$
for some arithmetic functions $\mathfrak{f}, \mathfrak{g}\colon \mathbb{N} \rightarrow \mathbb{C}$ such that $|\mathfrak{f}|, |\mathfrak{g}|\leq d_{10}$ and supported on $[N, 2 N]$ and $[M, 2 M]$ respectively, satisfying the bounds
$$
X^{\varepsilon^2} \ll N \ll HX^{-\varepsilon^2}\quad \text{and} \quad NM\asymp_{k, \varepsilon_0}X/q_0.
$$
\item Small sum: A function $\widetilde{\Lambda_f}$ supported on $(X/q_0, 2X/q_0]$ satisfying
\begin{align}\label{Eq: small sum}
    \|\widetilde{\Lambda_f}\|_{\ell^2}^2\ll_{k, \varepsilon} X^{1-\varepsilon^2/8}.
\end{align}
 \end{itemize}

So, it suffices to show that for $\widetilde{\Lambda_f}$ being one of the above forms and for any $(\log X)^{{3A}}/H<|\xi|\leq 1/(qQ)$, one has
\begin{equation}\label{eq:hybrid-moment}
\int_{\xi X/\sqrt{Q}}^{\xi X\sqrt{Q}}
\bigg(
  \sum_{\chi\Mod {q_1}}
    \int_{t-\xi H}^{\,t+\xi H}
      \bigg|\sum_{n}\dfrac{\widetilde{\Lambda_f}(n)\chi(n)}{n^{1/2+iu}}\bigg|du
\bigg)^{\!2}
\,dt \ll_{k, \: \varepsilon, \: A}\dfrac{q_1\xi^2H^2X}{(\log X)^{A+O
(1)}}.
\end{equation}

\medskip

\noindent{\bf Small sum:} Let us begin with the small sum case. By the Cauchy-Schwarz inequality, we have
\begin{align*}
    \bigg(
  \sum_{\chi\Mod {q_1}}
    \int_{t-\xi H}^{\,t+\xi H}
      \bigg|\sum_{n}\dfrac{\widetilde{\Lambda_f}(n)\chi(n)}{n^{1/2+iu}}\bigg|du\bigg)^2\ll q_1|\xi|H \sum_{\chi\Mod {q_1}}
    \int_{t-\xi H}^{\,t+\xi H}
      \bigg|\sum_{n}\dfrac{\widetilde{\Lambda_f}(n)\chi(n)}{n^{1/2+iu}}\bigg|^2\: du.
\end{align*}
This implies that after changing the order of integrals, the left-hand side of \eqref{eq:hybrid-moment} is
\begin{align*}
    &\ll q_1\xi^2H^2\sum_{\chi\Mod {q_1}}\int_{|t|\ll \xi X\sqrt{Q}}\bigg|\sum_{n}\dfrac{\widetilde{\Lambda_f}(n)\chi(n)}{n^{1/2+it}}\bigg|^2\: dt\\
    &\ll q_1\xi^2H^2 \cdot \dfrac{q_1\xi XQ^{1/2} + X/q_0}{X/q_0} \|\widetilde{\Lambda_f}\|_{\ell^2}^2(\log X)^3,
\end{align*}
where we have used Lemma \ref{Lemma: mean value} in the last line. By \eqref{Eq: small sum} together with $q_0, q_1\leq Q\ll (\log X)^{2A+O(1)}$, this establishes the claim \eqref{eq:hybrid-moment}.

\medskip

\noindent{\bf Type II sums:} We now handle the Type II sums. Without loss of generality, we may assume that \(N_{j}\ll M_{j} \).  Since $\widetilde{\Lambda_f}=\mathfrak{f}*\mathfrak{g}$, we have
\begin{align*}
    \sum_{n}\dfrac{\widetilde{\Lambda_f}(n)\chi(n)}{n^{1/2+iu}}=\bigg(\sum_{n}\dfrac{\mathfrak{f}(n)\chi(n)}{n^{1/2+iu}}\bigg)\bigg(\sum_{n}\dfrac{\mathfrak{g}(n)\chi(n)}{n^{1/2+iu}}\bigg).
\end{align*}
So, by the Cauchy-Schwarz inequality, we infer that
\begin{align*}
    \bigg(
  \sum_{\chi\Mod {q_1}}&
    \int_{t-\xi H}^{\,t+\xi H}
      \bigg|\sum_{n}\dfrac{\widetilde{\Lambda_f}(n)\chi(n)}{n^{1/2+iu}}\bigg|\: du
\bigg)^{\!2}\\
&\leq \bigg(
  \sum_{\chi\Mod {q_1}}
    \int_{t-\xi H}^{\,t+\xi H}
      \bigg|\sum_{n}\dfrac{\mathfrak{f}(n)\chi(n)}{n^{1/2+iu}}\bigg|^2\: du\bigg) \bigg(
  \sum_{\chi\Mod {q_1}}
    \int_{t-\xi H}^{\,t+\xi H}
      \bigg|\sum_{n}\dfrac{\mathfrak{g}(n)\chi(n)}{n^{1/2+iu}}\bigg|^2 du\bigg).
\end{align*}
Therefore, applying Lemma \ref{Lemma: mean value} and Fubini's theorem, we deduce that the left-hand side of \eqref{eq:hybrid-moment} is
\begin{align*}
    & \ll_{k,\varepsilon} q_1\bigg( Q^{1/2}q_1|\xi| +  \frac{Q^{1/2}N}{H}+\frac{1}{N}+\frac1{q_1|\xi|H}
\bigg)|\xi|^2H^2X(\log X)^{O_{\varepsilon}(1)}\\
&\ll \dfrac{q_1|\xi|^2H^2X}{(\log X)^{A+O(1)}},
\end{align*}
using the facts that \(X^{\varepsilon^2}\ll N \ll  HX^{-\varepsilon^2}\) and \(|\xi|\ll 1/(q_0q_1Q)\). This establishes the bounds for Type II sums.

\medskip

\noindent{\bf Type $d_1, d_2$ sums:} Finally, we handle the Type $d_1, d_2$ sums. Since $\widetilde{\Lambda_f
}=\mathfrak{f}*\mathfrak{g}_1*\mathfrak{g}_2$ in this case, we have
\begin{align*}
    \sum_{n}\dfrac{\widetilde{\Lambda_f}(n)\chi(n)}{n^{1/2+iu}}=\bigg(\sum_{n}\dfrac{\mathfrak{f}(n)\chi(n)}{n^{1/2+iu}}\bigg)\bigg(\sum_{n}\dfrac{\mathfrak{g}_1(n)\chi(n)}{n^{1/2+iu}}\bigg)\bigg(\sum_{n}\dfrac{\mathfrak{g}_2(n)\chi(n)}{n^{1/2+iu}}\bigg).
\end{align*}
By the Cauchy-Schwarz inequality, we have
\begin{align*}
     \bigg(
  \sum_{\chi\Mod {q_1}} &
    \int_{t-\xi H}^{\,t+\xi H}
      \bigg|\sum_{n}\dfrac{\widetilde{\Lambda_f}(n)\chi(n)}{n^{1/2+iu}}\bigg|\: du
\bigg)^{2}\\
\ll &\: \bigg( \sum_{\chi\Mod {q_1}}
    \int_{t-\xi H}^{\,t+\xi H}
      \bigg|\sum_{n}\dfrac{\mathfrak{f}(n)\chi(n)}{n^{1/2+iu}}\bigg|^2\: du\bigg)\\
& \times \: \bigg( \sum_{\chi\Mod {q_1}}
    \int_{t-\xi H}^{\,t+\xi H}
      \bigg|\sum_{n}\dfrac{\mathfrak{g}_1(n)\chi(n)}{n^{1/2+iu}}\bigg|^2\: du\bigg)\bigg( \sum_{\chi\Mod {q_1}}
    \int_{t-\xi H}^{\,t+\xi H}
      \bigg|\sum_{n}\dfrac{\mathfrak{g}_2(n)\chi(n)}{n^{1/2+iu}}\bigg|^2\: du\bigg).
\end{align*}
By Lemma \ref{Lemma: mean value}, we have
\begin{align*}
    \sum_{\chi\Mod {q_1}}
    \int_{t-\xi H}^{\,t+\xi H}
      \bigg|\sum_{n}\dfrac{\mathfrak{f}(n)\chi(n)}{n^{1/2+iu}}\bigg|^2\: du\ll (q_1|\xi|H + N)(\log X)^{O(1)}.
\end{align*}
So, we may now apply Fubini's theorem to infer that the left-hand side of \eqref{eq:hybrid-moment} is
\begin{align*}
    \ll_{\varepsilon} &\:  (q_1|\xi|H + N)|\xi|H(\log X)^{O(1)}\\
    & \times \: \sum_{\chi\Mod {q_1}}\int_{\xi X/\sqrt{Q}}^{\xi X\sqrt{Q}}\bigg|\sum_{n}\dfrac{\mathfrak{g}_1(n)\chi(n)}{n^{1/2+it}}\bigg|^2\bigg|\sum_{n}\dfrac{\mathfrak{g}_2(n)\chi(n)}{n^{1/2+it}}\bigg|^2\: dt.
\end{align*}
Observe that
\begin{align*}
\sum_{\chi\Mod {q_1}} & \int_{\xi X/\sqrt{Q}}^{\xi X\sqrt{Q}}\bigg|\sum_{n}\dfrac{\mathfrak{g}_1(n)\chi(n)}{n^{1/2+it}}\bigg|^2\bigg|\sum_{n}\dfrac{\mathfrak{g}_2(n)\chi(n)}{n^{1/2+it}}\bigg|^2\: dt\\
&\ll (\log X)\max_{\xi X/\sqrt{Q}\leq T\leq \xi X\sqrt{Q}}\sum_{\chi\Mod {q_1}}\int_{T/2}^T\bigg|\sum_{n}\dfrac{\mathfrak{g}_1(n)\chi(n)}{n^{1/2+it}}\bigg|^2\bigg|\sum_{n}\dfrac{\mathfrak{g}_2(n)\chi(n)}{n^{1/2+it}}\bigg|^2\: dt.
\end{align*}
Now we apply the Cauchy-Schwarz inequality and Hypothesis \ref{Hypothesis}, so that the above quantity is
\begin{align*}
    \ll \max_{\xi X/\sqrt{Q}\leq T\leq \xi X\sqrt{Q}} q_1T(\log X)^{O(1)} \ll q_1|\xi|XQ^{1/2}(\log X)^{O(1)}.
\end{align*}
Therefore, the total contribution to the left-hand side of $\eqref{eq:hybrid-moment}$ is
\begin{align*}
    \ll q_1(q_1|\xi|H + N)\xi^2 HX Q^{1/2}(\log X)^{O(1)} \ll_{A}  \dfrac{q_1\xi^2H^2X}{(\log X)^{A+O(1)}},
\end{align*}
using that $N\ll X^{\varepsilon^2}$, $H\geq X^{{1/3+\varepsilon}}$, and $|\xi|\leq 1/(q_0q_1Q)$. This establishes \eqref{eq:hybrid-moment} in the case of Type $d_1, d_2$ sums.
\end{proof}

\section{Proof of Theorem \ref{averagetheorem}}\label{sec: Proof of theorem averaged} 
In order to prove Theorem \ref{averagetheorem}, we shall use the Petersson trace formula. Recall the \emph{harmonic weights} of $f\in \mathcal{H}_k$ given by
\begin{align*}
    \omega_f=\dfrac{\Gamma(k-1)}{(4\pi)^{k-1}\langle\, f , f \,\rangle
},
\end{align*}
where $\langle\, f , f \,\rangle$ denotes the Petersson inner product.

\begin{lem}[Petersson trace formula]\label{Lemma: Petersson}
Let $k\geq 1$ be an even integer. Then, for any two integers $m, n\geq 1$, we have
\[
\sum_{f \in \mathcal{H}_{k}} \omega_f\lambda_f(n)\lambda_f(m)
= \mathbf{1}_{m=n} + \mathcal{E}(k,m,n),
\]
where
\[
\mathcal{E}(k,m,n)
\ll
\dfrac{(mn)^{1/4}(\log(3mn))^2 d_2((m,n))}{k^{1/2}}.
\]
Here, the implied constant is absolute.
\end{lem}

\begin{proof}
Apply \cite[Corollary 14.24]{IKbook} with $q=1$.
\end{proof}

\begin{lem}\label{alternativePerelli}
   Let $\varepsilon>0$ and let $x$ be sufficiently large. Let $k\geq 1$ be an even integer such that $k \gg x^{1+5\varepsilon}$. Then, for $x^{\varepsilon} \ll y \ll x^{1-\varepsilon}$, we have 
   \begin{align*}
       \sum_{f\in \mathcal{H}_k}\omega_f\bigg|\sum_{x<n\leq x+y}\lambda_f(n)\Lambda(n)e(n\alpha)\bigg|^2\ll_{\varepsilon} y^2x^{-\varepsilon/2}.
   \end{align*}
\end{lem}

\begin{proof}
    By expanding the square, we see that 
    \begin{align*}
      \sum_{f\in \mathcal{H}_k}\omega_f\bigg|\sum_{x<n\leq x+y}\lambda_f(n)\Lambda(n)e(n\alpha)\bigg|^2  =\sum_{x<n, m\leq x+y}\Lambda(n)\Lambda(m)e((n-m)\alpha)\sum_{f\in \mathcal{H}_k}\omega_f\lambda_f(n)\lambda_f(m).
    \end{align*}
Next, by Lemma \ref{Lemma: Petersson}, we infer that the right-hand side of the above expression is
    \begin{align*}
       & =\sum_{x<n\leq x+y}\Lambda(n)^2 + O\bigg(\dfrac{1}{\sqrt{k}}\sum_{x<n, m\leq x+y}(mn)^{1/4+\varepsilon}\bigg)\\
        &\ll y(\log x)^2 + \dfrac{y^2x^{1/2+2\varepsilon}}{\sqrt{k}}\ll y^2x^{-\varepsilon/2},
    \end{align*}
    by our assumption that $k\gg x^{1+5\varepsilon}$. This completes the proof.
\end{proof}

We are now ready to complete the proof of Theorem \ref{averagetheorem}. 

\begin{proof}[Proof of Theorem \ref{averagetheorem}]
We will follow the argument used in the proof of Theorem~\ref{unconditionalthm}. Indeed, by variants of  \eqref{Eq: Theorem unconditional split} and \eqref{Eq: Theorem unconditional final}, we have
\begin{align*}
    \sum_{f\in \mathcal{H}_k}\omega_fV_f(X; H)\ll HX(\log X)^2\max_{0\leq \alpha\leq 1}\sum_{f\in \mathcal{H}_k}\omega_f \int_{\alpha-1/(2H)}^{\alpha+1/(2H)}|S_{f, \: \Lambda}(\beta)|^2\: d\beta.
\end{align*}
Therefore, it suffices to show that
\begin{align*}
   \max_{0\leq \alpha\leq 1} \sum_{f\in \mathcal{H}_k}\omega_f\int_{\alpha-1/(2H)}^{\alpha+1/(2H)}|S_{f, \: \Lambda}(\beta)|^2\:d\beta \ll_{\varepsilon} X^{1-\varepsilon}.
\end{align*}
By Gallagher's lemma \cite[Lemma 1.9]{M1971}, the left-hand side of the above inequality  is
\begin{align*}
    \ll \dfrac{1}{H^2}\max_{0\leq \alpha\leq 1}\int_{X}^{2X}\sum_{f\in \mathcal{H}_k}\omega_f\bigg|\sum_{x<n\leq x+H}\lambda_f(n)\Lambda(n)e(n\alpha)\bigg|^2\: dx.
\end{align*}
Next, we apply Lemma \ref{alternativePerelli}, so that the above quantity is
\begin{align*}
    \ll_{\varepsilon} \dfrac{1}{H^2}\int_{X}^{2X}H^2x^{-\varepsilon}\: dx\ll_{\varepsilon} X^{1-\varepsilon},
\end{align*}
as desired.
\end{proof}

\section{Concluding remark}\label{sec: concluding}
 One may try to give an alternative proof of our results by adapting the arguments of Perelli and Pintz \cite{PP1992}, who treat the $2k$-twin problem. Their analysis of the minor arcs relies on a zero-free region together with zero-density estimates of the Dirichlet $L$-functions. To transfer their method to our setting, one would appeal to the zero-density bounds in \cite{K1998}, since the minor arc estimates of Perelli-Pintz require strong zero-density results for Dirichlet $L$-functions. Indeed, assume that, for $1/2 \le \sigma \le 1$,
\[
\sum_{\chi \,(\mathrm{mod}\, q)} N_f(\sigma,T,\chi)
  \ll_{k} (qT)^{V(1-\sigma)} \, (\log qT)^{O(1)}
\]
and
\[
\sum_{q \le Q} \ \sum_{\chi \,(\mathrm{mod}\, q)} N_f(\sigma,T,\chi)
  \ll_{k} (Q^{2}T)^{V(1-\sigma)} \, (\log QT)^{O(1)},
\]
for some $2 \le V \le 3$. Then, the contribution of zeros with real part $\sigma > 1 - 1/V$ will impose the restriction $H \gg Q \gg X^{1 - 4/(3V)}$, while the contribution of zeros with $1/2\leq  \sigma \le 1 -1/V$ forces $H \gg X^{(3V-2)/8}$.
Since we also require $Q \ll X^{1/2}$ to apply the argument in \cite{PP1992}, it follows that when $V>8/3$, no nontrivial bound is obtainable. Under the stronger hypothesis that $V=2$, one recovers a result in the range $H \gg X^{1/2}$. If, in addition, one assumes the slightly stronger zero-density bound for $1/2\leq \sigma \le 4/5$ as in \cite{K1998}, then one obtains the result in the range $H \gg X^{1/2}$. However, we are able to get $H\gg X^{1/3+\varepsilon}$ by assuming hypotheses that are weaker compared to strong zero-density estimates.

\section*{Acknowledgements}
The authors would like to thank Youness Lamzouri for his comments on an earlier version of the paper.

\bibliographystyle{plain}

\end{document}